\pdfoutput=1
\documentclass[10pt]{extarticle}

\usepackage[left=1in,right=1in,top=1in,bottom=1.3in]{geometry}
\usepackage{amsmath, amssymb, amsthm, latexsym, amsfonts, indentfirst, xcolor, mathtools, manfnt, microtype}
\usepackage[utf8]{inputenc} 
\usepackage[english]{babel}

\usepackage[breaklinks]{hyperref}
\usepackage[]{cleveref}
\hypersetup{
	colorlinks = true, 
	urlcolor = cyan, 
	linkcolor = teal, 
	citecolor = cyan 
}

\usepackage{tikz} 
\usetikzlibrary{patterns}
\usepackage{pgfplots}
\usepackage{float}

\let\OLDthebibliography\thebibliography
\renewcommand\thebibliography[1]{
	\OLDthebibliography{#1}
	\setlength{\parskip}{0pt}
	\setlength{\itemsep}{0pt plus 0.3ex}
}

\renewcommand{\le}{\leqslant}
\renewcommand{\ge}{\geqslant}

\newtheorem{Theorem}{Theorem}

\newtheorem{Corollary}{Corollary}

\newcommand{\R}{\mathbb R}

\newcommand{\N}{\mathbb N}

\newcommand{\x}{\mathbf x}
\newcommand{\y}{\mathbf y}

\newcommand{\M}{\mathbb M}

\begin{document}

\date{}
\title{A note on Borsuk's problem in Minkowski spaces}
\author{
	Andrei M. Raigorodskii\thanks{Moscow Institute of Physics and Technology, Moscow, Russia; Moscow State University, Moscow, Russia; Caucasus Mathematical Center, Adyghe State University, Maykop, Russia; Buryat State University, Ulan-Ude, Ruassia. Email:~\href{mailto:mraigor@yandex.ru}{\tt mraigor@yandex.ru}.}
	\and 
	Arsenii Sagdeev\thanks{Alfréd Rényi Institute of Mathematics, Budapest, Hungary; Moscow Institute of Physics and Technology,  Moscow, Russia. Email:~\href{mailto:sagdeevarsenii@gmail.com}{\tt sagdeevarsenii@gmail.com}.}
}

\maketitle

\begin{abstract}
	In 1993, Kahn and Kalai famously constructed a sequence of finite sets in $d$-dimensional Euclidean spaces that cannot be partitioned into less than $(1.203\ldots+o(1))^{\sqrt{d}}$ parts of smaller diameter. Their method works not only for the Euclidean, but for all $\ell_p$-spaces as well. In this short note, we observe that the larger the value of $p$, the stronger this construction becomes.
\end{abstract}

\section{Introduction}

For a $d$-dimensional (Minkowski) normed space $\M^d$, its \textit{Borsuk's number} $b(\M^d)$ is defined as the minimum $k \in \N$ such that every bounded non-singleton subset of $\M^d$ can be partitioned into $k$ parts of smaller diameter. This notion was introduced by Gr\"unbaum, who proved in~\cite{Grun57} that $b(\M^2) \le 4$ and that this inequality is strict unless the `unit circle' of the Minkowski plane $\M^2$ is a parallelogram. Boltyanskii and Gohberg~\cite{BG65} studied this problem for larger dimensions and conjectured that the inequality $b(\M^d) \le 2^d$ always holds. Currently the best known upper bound is $ b(\M^d) = O(2^dd\ln d)$, see \cite{RZ97}.

When proposing the notion of Borsuk's number, Gr\"unbaum was motivated by classical Borsuk's `conjecture' (see~\cite{Bor}) that every bounded non-singleton set in the Euclidean space $ \R^d $ could be partitioned into $ d+1 $ parts of smaller diameter. In the above notation, this conjecture states that $ b(\R^d) = d+1 $, since a regular simplex in $ \R^d $ does certainly require $ d+1 $ parts. Borsuk's conjecture is confirmed for $ d \le 3 $, see surveys \cite{Rai1, Rai2} for details. However, it was dramatically disproved in 1993 by Kahn and Kalai~\cite{KK93}, who found counterexamples for all $ d \ge 2015 $ and proved that, in general, $ b(\R^d) \ge (1.203\ldots+o(1))^{\sqrt{d}} $ as $d \to \infty$. Now it is known that $ b(\R^{64}) > 65 $ and that $ b(\R^d) \ge (1.2255\ldots+o(1))^{\sqrt{d}} $,  see \cite{Bon, Jen} and \cite{Rai3}, respectively. At the same time, it is shown that $ b(\R^d) \le 2^{d-1}+1 $ for all $d \in \N$, see~\cite{Las}, and that asymptotically $ b(\R^d) \le (1.224\ldots+o(1))^{d} $, see \cite{Sch, BL}. Note that the latter upper bound is much stronger than the general one from the Boltyanskii--Gohberg conjecture. For $ d \in [4,63] $, Borsuk's conjecture remains open. 

In this short note, we focus on the $d$-dimensional spaces $\ell_p^d$ equipped with the $\ell_p$-norm. (Recall that for $\x = (x_1,\dots,x_d) \in \R^d$, its \textit{$\ell_p$-norm} is given by $\|\x\|_p = (\sum_{i=1} \hspace{-7mm} \phantom{a}^d \hspace{3mm} |x_i|^p)^{1/p}$ for each real $p\ge 1$, and by $\|\x\|_\infty = \max_{i}|x_i|$ in case $p=\infty$.) It is easy to see that $b(\ell_\infty^d) = 2^d$ for all $d \in \N$. The Boltyanskii--Gohberg conjecture that $b(\M^d) \le 2^d$ was confirmed for all $\ell_p^d$ spaces in case $d=3$ by Yu and Zong~\cite{YZ09}, and very recently in case $d=4$ by Wang and Xue~\cite{WX22}. To the best of our knowledge, the conjecture remains open for each real $p\neq 2$, and $d>4$. As for the lower bounds on $b(\ell_p^d)$, we show that a modification of Kahn and Kalai's approach yields the following.

\begin{Theorem} \label{T_Main}
	Given $n,k \in \N$ and  $p, \lambda \in \R$ such that $1<k<n/2$, $p\ge 1$, $-1/2 \le \lambda \le 0$, put $d = \binom{n}{2}$,
	\begin{equation*}
		t_0 = \frac{(3k-n)|\lambda|^p+k|1+\lambda|^p+(1/2-k)|1+ 2\lambda|^p}{2|\lambda|^p+2|1+\lambda|^p-|1+2\lambda|^p}.
	\end{equation*}
	Let $t_1\le \lfloor t_0+1/2 \rfloor$ be the largest integer such that $k-t_1$ is a prime or a prime power. If $0<t_1 < k/2$, then
	\begin{equation*}
		b(\ell_p^d) \ge \binom{n}{k} \Big/ \binom{n}{k-t_1-1}.
	\end{equation*}
\end{Theorem}

For an arbitrary prime power $q$, applying the latter result with $n=4q-1, k=2q-1, \lambda=-1/2$ yields that $b(\ell_p^d) \ge \binom{4q-1}{2q-1}/\binom{4q-1}{q-1}$ regardless of the value of $p\ge 1$. For large $d$, this retrieves the original lower bound $b(\ell_p^d) \ge (3^{3/4}/2+o(1))^{\sqrt{2d}} = (1.203\ldots+o(1))^{\sqrt{d}}$ due to Kahn and Kalai. Though the aforementioned choice of the auxiliary parameters is asymptotically optimal in case of the Euclidean norm and, more generally, whenever $1\le p \le 2$, it is perhaps unexpected that for all $p>2$, we can do better. For the optimal choice of the parameters and for the resulting value of $c(p)$ such that $b(\ell_p^d) \ge (c(p)+o(1))^{\sqrt{d}}$ as $d \to \infty$, see \Cref{Tab2} and \Cref{Fig1}. To check that taking the parameters as in the first three columns of \Cref{Tab2} indeed leads to the bounds from its last column, one should just apply Stirling's approximation of binomial coefficients together with the observation that $t_1 = t_0-o(n)$ since the gaps between consecutive primes are known to be relatively small, see~\cite{BakHarPin2001}. It can be shown that the limit $\lim_{p \to \infty} c(p)$ of our lower bound equals $(\frac{1+\sqrt{2}}{2})^{\sqrt{2}} = 1.304\ldots$ This value represents a natural barrier that cannot be overcome within the current techniques.

There is an important subtlety concerning the last sentence of the previous paragraph. As we mentioned above, in case of $ \ell_2^d = \R^d $, there is a better bound $ b(\R^d) \ge (1.2255\ldots+o(1))^{\sqrt{d}} $. The corresponding approach is slightly different, and we did not succeed in extending it to $ p \neq 2 $. Nevertheless, even our current techniques provide us with much bigger constants $ c(p) $ than $ 1.2255 $, as it can be seen in \Cref{Tab2}.   

One can also discuss small-dimensional counterexamples to Borsuk's conjecture. In case of $ p=2 $, various modifications of Kahn and Kalai's approach led only to disproving the conjecture in dimensions $ d \ge 560 $, see \cite{Rai4, W}. Counterexamples in smaller dimensions are due to different techniques. For $ p \neq 2 $, we prove the following. 

\begin{Corollary} \label{Cor2}
	For all $p \ge 2.81$, we have $b(\ell_p^{406}) \ge 422$.
\end{Corollary}
\begin{proof}
	Apply \Cref{T_Main} with $n=29, k=9, \lambda = -1/3$. Note that $t_0=9/2-15/(2^{p+1}+2)$, which is no less than $7/2$ whenever $p\ge2.81$, and thus $t_1=4$. Now \Cref{T_Main} yields that for $d=\binom{29}{2}=406$, we have $b(\ell_p^d) \ge \binom{29}{9}/\binom{29}{4} > 421$, as desired.
\end{proof}

We do not assert that this result is best possible, since the bound we use in the last step of the proof of \Cref{T_Main} is not tight. One can try to use instead slightly better bounds similar to those in \cite{Bog} and \cite[Section~4]{Diss}. This will not provide any asymptotic improvement, but in concrete dimensions, this might work. 

\section{Proof of Theorem~\ref{T_Main}}

Let $V$ be the set of all $\binom{n}{k}$ points in $\{0,1\}^n$ with $k$ unit coordinates. 
We map every $\x=(x_1,\dots,x_n)\in V$ to $\x^* = (x_{i,j})_{1\le i <j \le n}\in \ell_p^d$
defined by $x_{i,j} = x_ix_j+\lambda x_i +\lambda x_j$. Note that each $\x^*$ has precisely $\binom{n-k}{2}$, $k(n-k)$, and $\binom{k}{2}$ coordinates equal to $0$, $\lambda$, and $1+2\lambda$, respectively. So, for all $\x, \y \in V$, each of the $d=\binom{n}{2}$ pairs $(i,j)$ falls into one of the $3^2=9$ groups based one the corresponding values of $x_{i,j}$ and $y_{i,j}$. Moreover, if we denote the number of common units of $\x$ and $\y$ by $t$, then it is not hard to express the sizes of these groups, see \Cref{Tab1}. 

\begin{table}[h]
	\centering
	\begin{tabular}{|c|c|c|c|}
		\hline
		& $x_{i,j}=0$ & $x_{i,j}=\lambda$ & $x_{i,j}=1+2\lambda$  \\ \hline
		\hspace{-10mm} $y_{i,j}=0$ & $\binom{n-2k+t}{2}^{\phantom{l}}_{\phantom{q}}$ & $(n - 2k + t)(k - t)$ & $\binom{k-t}{2}$ \\ \hline
		\hspace{-10mm} $y_{i,j}=\lambda$ & $(n - 2k + t)(k - t)$ & $(k - t)^{2^{\phantom{l}}} + t(n - 2k + t)$ & $t(k - t)$ \\ \hline
		$y_{i,j}=1+2\lambda$ & $\binom{k-t}{2}$ & $t(k - t)$ & $\binom{t}{2}^{\phantom{l}}_{\phantom{q}}$ \\ \hline
	\end{tabular}
	\caption{The number of pairs $(i,j)$ with the prescribed values of $x_{i,j}$ and $y_{i,j}$.}
	\label{Tab1}
\end{table}

Now a routine calculation yields that 
\begin{align*}
	\|\x^*-\y^*\|_p^p = at^2 + bt + c,\ \mbox{ where } \
	  a&= -2|\lambda|^p-2|1+\lambda|^p+|1+2\lambda|^p, \\
	  b&= 2(3k-n)|\lambda|^p+2k|1+\lambda|^p+(1-2k)|1+2\lambda|^p, \\
	  c&= 2k(n-2k)|\lambda|^p+k(k-1)|1+2\lambda|^p.
\end{align*}

It is clear that $|1+\lambda|^p\ge|1+2\lambda|^p$, and thus $a<0$. Therefore, as a function of a real variable $t$, the parabola $at^2 + bt + c$ is concave and reaches its maximum at $t=t_0 \coloneqq -b/(2a)$. Moreover, note that $c>0$. Since the distance from each $\x^*$ to itself trivially equals $0$, one of the parabolas roots is $t=k$, while the other one is $t=c/(ka) < 0$. In particular, this implies that our mapping $\x \to \x^*$ is injective on $V$. 

Next, we observe that as a function of an \textit{integer} variable, this parabola reaches its maximum at the closest integer to the vertex $t=t_0$, namely at $t=\lfloor t_0+1/2 \rfloor$. We claim that if the latter value is positive, then we can increase $\lambda$ to ensure that the distance $\|\x^*-\y^*\|_p$ is maximum whenever $\x$ and $\y$ share precisely $t_1$ common units. Indeed, increasing $\lambda$ all the way to $0$ results in decreasing the parabola vertex $-b/(2a)$ to $1/2 < t_1$, and thus our claim follows from the intermediate value theorem. Finally, the next classic result ensures that every sufficiently large subset of $V$ contains two such points.

\begin{Theorem}[Frankl--Wilson, \cite{FranklWil1981}] \label{th:FW}
	Given $n,k,t \in \N$ such that $2t<k<n/2$ and $k-t$ is a prime or a prime power, let $V$ be the set of all $\binom{n}{k}$ points in $\{0,1\}^n$ with $k$ unit coordinates. Then every subset of more than $\binom{n}{k-t-1}$ elements of $V$ contains two points sharing  precisely $t$ common units.
\end{Theorem}

Now the pigeonhole principle implies that the set $\{\x^*: \x \in V\} \subset \ell_p^d$ cannot be partitioned into less than $\binom{n}{k}/\binom{n}{k-t_1-1}$ parts of smaller diameter, which completes the proof.

\vskip+0.4cm

Note that the proof is based on `$ (0,1) $-vectors', which we map into vectors, whose coordinates may have 3 different values. The result of~\cite{Rai3} is due to changing $ (0,1) $-vectors by $ (-1,0,1) $-vectors. For the moment, we see a big technical problem in mapping such vectors into $ \ell_p^d $ to implement a similar idea as the one in the just-given proof. If the problem is solved, better lower bounds on $b(\ell_p^d)$ for $ p \neq 2 $ may also appear.    

\section*{Acknowledgments} We thank Olga Kostina for the helpful discussion. The first author is supported by the grant NSh-775.2022.1.1. The second author is supported by ERC Advanced Grant `GeoScape' No. 882971.

\appendix

\section{Numerical data}

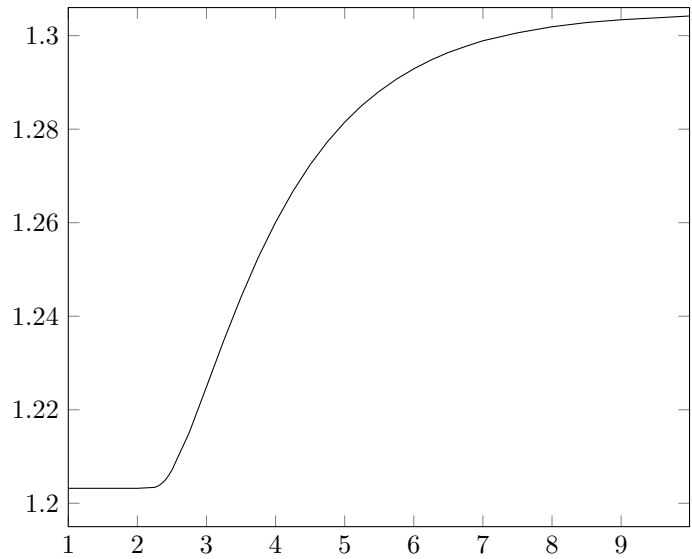
\begin{figure}[h!]
	\centering
	\begin{tikzpicture}
		\begin{axis}[
			width = 280,
			xmin = 1,
			xmax = 9.99,
			ymin = 1.195,
			ymax = 1.306]
			
			\addplot[] table {
				x		y
				1.00 1.2032
				2.00 1.2032
				2.25 1.2034
				2.30 1.2037
				2.35 1.2042
				2.40 1.2049
				2.45 1.2059
				2.50 1.2071
				2.75 1.2151
				3.00 1.2249
				3.25 1.2348
				3.50 1.2441
				3.75 1.2526
				4.00 1.2601
				4.25 1.2667
				4.50 1.2724
				4.75 1.2773
				5.00 1.2815
				5.25 1.2851
				5.50 1.2881
				5.75 1.2907
				6.00 1.2929
				6.25 1.2948
				6.50 1.2964
				6.75 1.2977
				7.00 1.2989
				7.50 1.3006
				8.00 1.3019
				8.50 1.3028
				9.00 1.3034
				9.50 1.3038
				9.99 1.3042
			};
		\end{axis}
	\end{tikzpicture}
	\caption{The graph of the function $c(p)$ for $1\le p \le 10$.}
	\label{Fig1}
\end{figure}

\begin{table}[h]
	\centering
	\begin{tabular}{|c|c|c|c|c|}
		\hline
		$p$ & $-\lambda$ & $k/n$ & $t_0/n$ & $c(p)$  \\ \hline
		1.00 & 0.5000 & 0.5000 & 0.2500 & 1.2032 \\ \hline
		2.00 & 0.5000 & 0.5000 & 0.2500 & 1.2032 \\ \hline
		2.25 & 0.4639 & 0.4777 & 0.2287 & 1.2034 \\ \hline
		2.30 & 0.4472 & 0.4666 & 0.2189 & 1.2037 \\ \hline
		2.35 & 0.4310 & 0.4554 & 0.2095 & 1.2042 \\ \hline
		2.40 & 0.4163 & 0.4448 & 0.2012 & 1.2049 \\ \hline
		2.45 & 0.4031 & 0.4348 & 0.1938 & 1.2059 \\ \hline
		2.50 & 0.3915 & 0.4255 & 0.1874 & 1.2071 \\ \hline
		2.75 & 0.3521 & 0.3895 & 0.1665 & 1.2151 \\ \hline
		3.00 & 0.3317 & 0.3656 & 0.1562 & 1.2249 \\ \hline
		3.25 & 0.3207 & 0.3491 & 0.1510 & 1.2348 \\ \hline
		3.50 & 0.3146 & 0.3372 & 0.1483 & 1.2441 \\ \hline
		3.75 & 0.3112 & 0.3283 & 0.1468 & 1.2526 \\ \hline
		4.00 & 0.3095 & 0.3215 & 0.1460 & 1.2601 \\ \hline
		4.25 & 0.3087 & 0.3162 & 0.1456 & 1.2667 \\ \hline
		4.50 & 0.3085 & 0.3120 & 0.1455 & 1.2724 \\ \hline
		4.75 & 0.3087 & 0.3087 & 0.1455 & 1.2773 \\ \hline
		5.00 & 0.3091 & 0.3060 & 0.1455 & 1.2815 \\ \hline
		5.25 & 0.3097 & 0.3038 & 0.1456 & 1.2851 \\ \hline
		5.50 & 0.3104 & 0.3020 & 0.1457 & 1.2881 \\ \hline
		5.75 & 0.3110 & 0.3005 & 0.1458 & 1.2907 \\ \hline
		6.00 & 0.3118 & 0.2993 & 0.1459 & 1.2929 \\ \hline
		6.25 & 0.3125 & 0.2982 & 0.1459 & 1.2948 \\ \hline
		6.50 & 0.3132 & 0.2974 & 0.1460 & 1.2964 \\ \hline
		6.75 & 0.3138 & 0.2967 & 0.1461 & 1.2977 \\ \hline
		7.00 & 0.3145 & 0.2961 & 0.1462 & 1.2989 \\ \hline
		7.50 & 0.3157 & 0.2951 & 0.1462 & 1.3006 \\ \hline
		8.00 & 0.3168 & 0.2945 & 0.1463 & 1.3019 \\ \hline
		8.50 & 0.3179 & 0.2940 & 0.1463 & 1.3028 \\ \hline
		9.00 & 0.3188 & 0.2937 & 0.1464 & 1.3034 \\ \hline
		9.50 & 0.3196 & 0.2935 & 0.1464 & 1.3038 \\ \hline
		9.99 & 0.3203 & 0.2933 & 0.1464 & 1.3042 \\ \hline
	\end{tabular}
	\caption{The optimal choice of the auxiliary parameters and the resulting values of $c(p)$}
	\label{Tab2}
\end{table}


{\small \begin{thebibliography}{20}

\bibitem{BakHarPin2001} R.C. Baker, G. Harman, J. Pintz, {\it The difference between consecutive primes. II}, Proc. Lond. Math. Soc. (3), \textbf{83}.3 (2001), 532--562.

\bibitem{Bog} L.I. Bogolubsky, A.M. Raigorodskii, {\it On bounds in Borsuk’s problem}, Proc. MIPT (Trudy MFTI), {\bf 11}.3 (2019), 20--49.

\bibitem{BG65} V.G. Boltyanskii, I.T. Gohberg, \textit{Results and Problems in Combinatorial Geometry}, Cambridge Univ. Press, Cambridge, 1985; Nauka, Moscow, 1965.

\bibitem{Bon} A. Bondarenko, \textit{On Borsuk’s Conjecture for Two-Distance Sets}, Discrete Comput. Geom., \textbf{51}.3 (2014), 509--515.

\bibitem{Bor} K. Borsuk, {\it Drei S\"atze \"uber die
$n$-dimensionale euklidische Sph\"are}, Fundamenta Math., \textbf{20} (1933),
177--190.

\bibitem{BL} J. Bourgain, J. Lindenstrauss, {\it On covering a set in $ {\mathbb R}^d $ by balls of the same diameter}, Geometric Aspects of Functional Analysis (J. Lindenstrauss and V. Milman, eds.), Lecture Notes in
Math., 1469, Springer, Berlin, 1991, 138--144.

\bibitem{FranklWil1981} P. Frankl, R.M. Wilson, {\it Intersection theorems with geometric consequences}, Combinatorica, \textbf{1}.4 (1981), 357--368.

\bibitem{Grun57} B. Gr\"unbaum, \textit{Borsuk’s partition conjecture in Minkowski planes}, Bull. Res. Council Israel Sect. F, 7F (1957), 25--30.

\bibitem{Jen} T. Jenrich, A.E. Brouwer, \textit{A 64-Dimensional Counterexample to Borsuk’s Conjecture}, Electron. J. Combin., \textbf{21}.4 (2014), P4.29, 3 pp.

\bibitem{KK93} J. Kahn, G. Kalai, \textit{A counterexample to Borsuk’s conjecture}, Bull. Amer. Math. Soc., \textbf{29}.1 (1993), 60--62.

\bibitem{Las} M. Lassak, {\it An estimate concerning Borsuk's
partition problem}, Bull. Acad. Polon. Sci. Ser. Math., \textbf{30} (1982), 449--451.

\bibitem{Diss} A. Passuello, {\it Semidefinite programming in combinatorial optimization with applications to coding theory and geometry}, 2013, thesis, available at https://theses.hal.science/tel-00948055. 

\bibitem{Rai1} A.M. Raigorodskii, {\it Around Borsuk's conjecture}, 
J. Math. Sci., \textbf{154}.4 (2008), 604--623.

\bibitem{Rai2} A.M. Raigorodskii, {\it Coloring Distance Graphs and Graphs of Diameters}, Thirty Essays on Geometric Graph Theory, J. Pach ed., Springer, 2013, 429--460.

\bibitem{Rai3} A.M. Raigorodskii, {\it On a bound in Borsuk's problem}, Russian Math. Surveys, \textbf{54}.2 (1999), 453--454.

\bibitem{Rai4} A.M. Raigorodskii, {\it On the dimension in Borsuk's problem},
Russian Math. Surveys, {\bf 52}.6 (1997), 1324--1325.

\bibitem{RZ97} C.A. Rogers, C. Zong, \textit{Covering convex bodies by translates of convex bodies}, Mathematika,  \textbf{44}.1
(1997), 215--218.

\bibitem{Sch} O. Schramm, {\it Illuminating sets of constant width}, Mathematika, \textbf{35} (1988), 180--189.

\bibitem{Swan18} K.J. Swanepoel, \textit{Combinatorial distance geometry in normed spaces}, New Trends in Intuitive Geometry (2018), 407--458.

\bibitem{WX22} J. Wang, F. Xue, \textit{Borsuk's partition problem in four-dimensional $\ell_{p}$ space}, preprint arXiv:2206.15277.

\bibitem{W} B. Weissbach, {\it Sets with large Borsuk number}, Beitr\"age Alg. Geom., {\bf 41} (2000), 417--423.

\bibitem{YZ09} L. Yu, C. Zong, \textit{On the blocking number and the covering number of a convex body}, Adv. Geom., \textbf{9}.1 (2009), 13--29.

\end{thebibliography}}
\end{document}